\documentclass[graybox]{svmult}

\smartqed
\usepackage{mathptmx}       
\usepackage{helvet}         
\usepackage{courier}        
\usepackage{type1cm}        
\usepackage{graphicx}       

\usepackage{array,colortbl}
\usepackage{amsmath,amsfonts,amssymb,bm} 
\DeclareFontFamily{U}{mathx}{\hyphenchar\font45}
\DeclareFontShape{U}{mathx}{m}{n}{
      <5> <6> <7> <8> <9> <10>
      <10.95> <12> <14.4> <17.28> <20.74> <24.88>
      mathx10
      }{}
\DeclareSymbolFont{mathx}{U}{mathx}{m}{n}
\DeclareFontSubstitution{U}{mathx}{m}{n}
\DeclareMathAccent{\widecheck}      {0}{mathx}{"71}

 \newcommand{\Z}{\mathbb{Z}} 
\newcommand{\R}{\mathbb{R}} 
\newcommand{\T}{\mathbb{T}} 

\newcommand{\rd}{\,\mathrm{d}} 
\newcommand{\bsv}{\boldsymbol{v}}    
\newcommand{\bsw}{\boldsymbol{w}}    
\newcommand{\bsx}{\boldsymbol{x}}    
\newcommand{\bsy}{\boldsymbol{y}}    
\newcommand{\bsB}{\boldsymbol{B}}    

\newcommand{\bsH}{\boldsymbol{H}} 
\newcommand{\bslambda}{\boldsymbol{\lambda}}

\newcommand{\bsmu}{\boldsymbol{\mu}} 

\DeclareSymbolFont{bbold}{U}{bbold}{m}{n}
\DeclareSymbolFontAlphabet{\mathbbold}{bbold}

\newcommand{\cH}{\ensuremath{\mathcal{H}}}
\newcommand{\cL}{\ensuremath{\mathcal{L}}}

\newcommand{\cP}{\ensuremath{\mathcal{P}}}
\newcommand{\satop}[2]{\stackrel{\scriptstyle{#1}}{\scriptstyle{#2}}}

\DeclareMathOperator*{\argmin}{arg\,min}

\usepackage{microtype} 

\usepackage[colorlinks=true,linkcolor=black,citecolor=black,urlcolor=black]{hyperref}
\usepackage{bookmark}
\makeatletter 
  \providecommand*{\toclevel@author}{999}
  \providecommand*{\toclevel@title}{0}
\makeatother

\usepackage[english,ruled]{algorithm2e}

\begin{document}

\title*{Optimal point sets for quasi--Monte Carlo integration  of bivariate periodic functions with bounded mixed derivatives}
\titlerunning{Optimal point sets for quasi--Monte Carlo integration  of bivariate periodic functions}
\author{Aicke Hinrichs \and Jens Oettershagen}
\institute{
Aicke Hinrichs 
\at Institut f\"ur Analysis, Johannes-Kepler-Universit\"at Linz, Altenberger Stra\ss e 69, 4040 Linz, Austria \\
\email{aicke.hinrichs@uni-rostock.de} 
\and 
Jens Oettershagen 
\at Institute for Numerical Simulation, Wegelerstra\ss e 6, 53115 Bonn, Germany \\
\email{oettershagen@ins.uni-bonn.de}
}
\maketitle

\abstract{We investigate quasi-Monte Carlo (QMC) integration of bivariate periodic functions with dominating mixed smoothness of order one. While there exist several QMC constructions which asymptotically yield the optimal rate of convergence of $\mathcal{O}(N^{-1}\log(N)^{\frac{1}{2}})$, it is yet unknown which point set is optimal in the sense that it is a global minimizer of the worst case integration error.
We will present a computer-assisted proof by exhaustion  that the Fibonacci lattice is the unique minimizer of the QMC worst case error in periodic $H^1_\text{mix}$ for small Fibonacci numbers $N$. Moreover, we investigate the situation for point sets whose cardinality $N$ is not a Fibonacci number. It turns out that for $N=1,2,3,5,7,8,12,13$ the optimal point sets are integration lattices.}

\section{Introduction}

Quasi-Monte Carlo (QMC) rules are equal-weight quadrature rules which can be used to approximate integrals defined on the $d$-dimensional unit cube $[0,1)^d$
\begin{equation*}
\int_{[0,1)^d} f(\bsx) \,\mathrm{d} \bsx \approx \frac{1}{N} \sum_{i=1}^{N} f(\bsx_i),
\end{equation*}
where $\cP_N=\{\bsx_1,\bsx_2,\ldots,\bsx_{N}\}$ are deterministically chosen quadrature points in $[0,1)^d$.
The integration error for a specific function $f$ is given as
\begin{equation*}
\left| \int_{[0,1)^d} f(\bsx) \,\mathrm{d} \bsx - \frac{1}{N} \sum_{i=1}^{N} f(\bsx_i) \right|.
\end{equation*}
To study the behavior of this error as $N$ increases for $f$ from a Banach space $(\cH, \|\cdot\|)$ one considers the worst case error
\begin{equation*}
{\rm wce}(\cH,\cP_N)=\sup_{\satop{f \in \cH}{\|f\| \le 1}} \left| \int_{[0,1)^d} f(\bsx) \,\mathrm{d} \bsx - \frac{1}{N} \sum_{i=1}^{N} f(\bsx_i) \right|.
\end{equation*}
Particularly nice examples of such function spaces are reproducing kernel Hilbert spaces \cite{Aronszajn_1950}. 
Here, we will consider the reproducing kernel Hilbert space $H^1_\text{mix}$ of 1-periodic functions with mixed smoothness. Details on these spaces are given in Section 2. The reproducing kernel is a
tensor product kernel of the form
$$ K_{d,\gamma}(\bsx,\bsy) =  \prod_{j=1}^d K_{1,\gamma} (x_j,y_j) \ \mbox{for} \ \bsx=(x_1,\dots,x_d),\bsy=(y_1,\dots,y_d) \in [0,1)^d$$
with
$ K_{1,\gamma}(x,y)= 1 + \gamma k(|x - y|)$ and $k(t)=\frac{1}{2}(t^2-t+\frac{1}{6})$ and a parameter $\gamma>0$.
It turns out that minimizing the worst case error ${\rm wce}(H^1_\text{mix},\cP_N)$ among all $N$-point sets  $\cP_N=\{\bsx_1,\ldots,\bsx_{N}\}$ with respect to the Hilbert space norm 
corresponding to the kernel  $K_{d,\gamma}$ is equivalent to minimizing the double sum
$$ G_\gamma ( \bsx_1,\ldots,\bsx_{N} ) = \sum_{i,j=1}^N K_{d,\gamma}(\bsx_i,\bsx_j).$$
There is a general connection between the discrepancy of a point set and the worst case error of integration. Details can be found in \cite[Chapter 9]{NW10}. 
In our case, the relevant notion is the $L_2$-norm of the periodic discrepancy. 
We describe the connection in detail in Section \ref{sec_dis}.

There are many results on the rate of convergence of worst case errors and of the optimal discrepancies for $N\to \infty$, see e.g. \cite{Niederreiter,NW10}, 
but results on the optimal point configurations for fixed $N$ and $d>1$ are scarce.
For discrepancies, we are only aware of \cite{W77}, where the point configurations minimizing the standard  $L_\infty$-star-discrepancy for $d=2$ and $N=1,2,\dots,6$
are determined, \cite{PVC06}, where for $N=1$ the point minimizing the standard  $L_\infty$- and $L_2$-star discrepancy for $d \ge 1$ is found, and 
\cite{LP07}, where this is extended to $N=2$.

It is the aim of this paper to provide a method which for $d=2$ and $N>2$ yields the optimal points for the periodic $L_2$-discrepancy and worst case error in $H^1_\text{mix}$. 
Our approach is based on a decomposition of the global optimization problem into exponentially many local ones which each possess unique solutions that can be approximated efficiently by a nonlinear block Gau\ss-Seidel method. Moreover, we use the symmetries of the two-dimensional torus to significantly reduce the number of local problems that have to be considered.

It turns out that in the case that $N$ is a (small) Fibonacci number, the Fibonacci lattice yields the optimal point configuration.
It is common wisdom, see e.g. \cite{BTY12,NS84,SJ94,SZ82}, that the Fibonacci lattice provides a very good point set for integrating periodic functions. Now our results support the
conjecture that they are actually the best points.  

These results may suggest that the optimal point configurations are integration lattices or at least lattice point sets. This seems to be true for some numbers $N$ of points, for example
for Fibonacci numbers, but not always. However, it can be shown that integration lattices are always {\em local} minima of ${\rm wce}(H^1_\text{mix},\cP_N)$. 
Moreover, our numerical results also suggest that for small $\gamma$ the optimal points are always \emph{close} to a lattice point set, i.e. $N$-point sets of the form
$$ \left\{  \left( \frac{i}{N}, \frac{\sigma(i)}{N} \right) \,:\, i=0,\dots,N-1\right\} ,$$
where $\sigma$ is a permutation of $\{0,1,\dots,N-1\}$. 

The remainder of this article is organized as follows: In Section 2 we recall Sobolev spaces with bounded mixed derivatives, the notion of the worst case integration error in reproducing kernel Hilbert spaces and the connection to periodic discrepancy. 
In  Section 3 we discuss necessary and sufficient conditions for optimal point sets and derive lower bounds of the worst case error on certain local patches of the whole $[0,1)^{2N}$. In Section 4 we compute candidates for optimal point sets up to machine precision. Using arbitrary precision rational arithmetic we prove that they are indeed near the global minimum which also turns out to be unique up to torus-symmetries. 
For certain point numbers the global minima are integration lattices as is the case if $N$ is a Fibonacci number.
We close with some remarks in Section 5.

\section{Quasi--Monte Carlo Integration in $H^1_\text{mix}(\mathbb{T}^2)$}

\subsection{Sobolev Spaces of Periodic Functions}
We consider univariate 1-periodic functions $f: \R \rightarrow \R$ which are given by their values on the torus $\T = [0,1)$.
For $k\in\Z$, the $k$-th Fourier coefficient of a function $f \in L_2(\T)$ is given by $\hat{f}_k = \int_0^1 f(x) \exp(2\pi \mathrm{i} \, k x) \, \rd x$. 
The definition
\begin{equation}
	\|f\|_{H^{1,\gamma}}^2 = \hat{f}_0^2 + \gamma \sum_{k \in \Z} |2\pi k|^2 \hat{f}_k^2 =  \left(\int_{\T} f(x) \, \rd x\right)^2 + \gamma \int_{\T} f'(x)^2 \, \rd x
\end{equation}
for a function $f$ in the univariate Sobolev space $H^1(\T) = W^{1,2}(\T) \subset L_2(\T)$ of functions with first weak derivatives bounded in $L_2$
gives a Hilbert space norm $\|f\|_{H^{1,\gamma}}$ on $H^1(\T)$ depending on the parameter $\gamma>0$.
The corresponding inner product is given by
\[
	(f,g)_{H^{1,\gamma}(\mathbb{T})} = \left( \int_0^1 f(x)\, \rd x \right) \left( \int_0^1  g(x) \, \rd x \right) + \gamma \int_0^1 f'(x) g'(x) \, \rd x .
\]
We denote the Hilbert space $H^1(\T)$ equipped with this inner product by $H^{1,\gamma}(\T)$.

Since $H^{1,\gamma}(\T)$ is continuously embedded in $C^0(\T)$ it is a reproducing kernel Hilbert space (RKHS), see \cite{Aronszajn_1950}, with a symmetric and positive definite kernel \linebreak
$K_{1,\gamma}: \T \times \T \rightarrow \R$, given by  \cite{Wahba75}
\begin{equation}
    \begin{aligned}
	K_{1,\gamma}(x,y) := & 1 + \gamma \sum_{k \in \mathbb{Z} \setminus \{0\}} |2\pi k|^{-2} \exp(2\pi \mathrm{i} k (x-y)) \\
		  = & 1+ \gamma k(|x-y|) ,
	\end{aligned}
\end{equation}
where $k(t) = \frac{1}{2}(t^2-t+\frac{1}{6})$ is the Bernoulli polynomial of degree two divided by two.

This kernel has the property that it reproduces point evaluations in $H^1$, i.e. \linebreak $f(x) = (f(\cdot),K(\cdot, x))_{H^{1,\gamma}}$ for all $f \in H^1$.
The reproducing kernel of the tensor product space $H^{1,\gamma}_\text{mix}(\T^2) := H^1(\mathbb{T}) \otimes  H^1(\mathbb{T}) \subset C(\mathbb{T}^2)$ is the product of the univariate kernels, i.e.
\begin{equation}
    \begin{aligned}
	K_{2,\gamma}(\bsx,\bsy) = & K_{1,\gamma}(x_1,y_1) \cdot K_{1,\gamma}(x_2,y_2) \\
		       = & 1 + \gamma k(|x_1 - y_1|) + \gamma k(|x_2 - y_2|) + \gamma^2 k(|x_1 - y_1|) k(|x_2 - y_2|) .
	\end{aligned}
\end{equation}

\subsection{Quasi--Monte Carlo Cubature} \label{sec_wce}
A linear cubature algorithm $Q_N(f) := \frac{1}{N} \sum_{i=1}^{N} f(\bsx_i)$ with uniform weights $\frac{1}{N}$ on a point set $\cP_N=\{\bsx_1,\dots,\bsx_{N}\}$ is called a QMC cubature rule.
Well-known examples for point sets used in such quadrature methods are digital nets, see e.g. \cite{DickPillich, Niederreiter}, and lattice rules \cite{SJ94}. 
A two-dimensional integration lattice is a set of $N$ points given as
$$ \left\{  \left( \frac{i}{N}, \frac{i g}{N} \mod 1 \right) \,:\, i=0,\dots,N-1\right\} $$
for some $g\in \{1,\dots,N-1\}$ coprime to $N$.     
A special case of such a rank-1 lattice rule is the so called Fibonacci lattice that only exists for $N$ being a Fibonacci number $F_n$ and is given by the generating vector $(1,g) = (1,F_{n-1})$, where $F_n$ denotes the $n$-th Fibonacci number. It is well known that the Fibonacci lattices yield the optimal rate of convergence in certain spaces of periodic functions.

In the setting of a reproducing kernel Hilbert space with kernel $K$ on a general domain $D$, the worst case error of the
QMC-rule $Q_N$ can be computed as
$$ {\rm wce}(\cH,\cP_N)^2 =  \int_{D} \int_{D} K(\bsx,\bsy) \, \rd \bsx \rd \bsy - \frac{2}{N}\sum_{i=1}^N \int_{D} K(\bsx_i,\bsy) \, \rd y + \frac{1}{N^2} \sum_{i,j=1}^{N} K(\bsx_i,\bsx_j) ,$$
which is the norm of the error functional, see e.g. \cite{DickPillich,NW10}.
For the kernel $K_{2,\gamma}$ we obtain
$$ {\rm wce}( H^{1,\gamma}_\text{mix}(\T^2) ,\cP_N)^2 = - 1 + \frac{1}{N^2} \sum_{i=1}^N \sum_{j=1}^N K_{2,\gamma}(\bsx_i, \bsx_j).$$
There is a close connection between the worst case error of integration in \linebreak ${\rm wce}( H^{1,\gamma}_\text{mix}(\T^2) ,\cP_N)$ for the case $\gamma=6$ and periodic $L_2$-discrepancy, which we will describe in the following.

\subsection{Periodic Discrepancy} \label{sec_dis}
The periodic $L_2$-discrepancy is measured with respect to periodic boxes. In dimension $d=1$, periodic intervals $I(x,y)$ for $x,y\in [0,1)$ are given by
$$ I(x,y)=[x,y) \ \mbox{if} \ x\le y \qquad \mbox{and} \qquad I(x,y)=[x,1) \cup [0,y) \ \mbox{if} \ x > y. $$
In dimension $d>1$, the periodic boxes $B(\bsx,\bsy)$ for $\bsx=(x_1,\dots,x_d)$ and  $\bsy=(y_1,\dots,y_d)\in [0,1)^d$ are products of the one-dimensional intervals, i.e.
$$ B(\bsx,\bsy) = I(x_1,y_1) \times \dots \times I(x_d,y_d).$$
The discrepancy of a set $\cP_N=\{ \bsx_1,\dots,\bsx_N \} \subset [0,1)^d$ with respect to such a periodic box $B=B(\bsx,\bsy)$ is the deviation of the relative number of points of $\cP_N$ in $B$
from the volume of $B$
$$  D({\cP}_N, B ) = \frac{\# {\cP}_N \cap B }{N}  -  \, {\rm vol} (B). $$
Finally, the periodic $L_2$-discrepancy of $\cP_N$ is the $L_2$-norm of the discrepancy function taken over all periodic boxes $B=B(\bsx,\bsy)$, i.e.
$$  D_2(\cP_N) = \left( \int_{[0,1)^d} \int_{[0,1)^d}  D({\cal P}_N, B(\bsx,\bsy) )^2 \rd \, \bsy \rd \bsx \right)^{1/2}. $$
It turns out, see \cite[page 43]{NW10} that the periodic $L_2$-discrepancy can be computed as 
\begin{align*}
	D_2 ({\cP}_N)^2 = & -3^{-d} + \frac{1}{N^2} \sum_{\bsx,\bsy \in  {\cal P}_N} \tilde{K}_d(\bsx,\bsy) \\
			= &  3^{-d} {\rm wce}( H^{1,6}_\text{mix}(\T^d) ,\cP_N)^2,
\end{align*}
where $\tilde{K}_d$ is the tensor product of $d$ kernels $\tilde{K}_1(x,y)= |x-y|^2 -|x-y| + \frac12$.
So minimizing the periodic $L_2$-discrepancy is equivalent to minimizing the worst case error in $H^{1,\gamma}_\text{mix}$ for $\gamma=6$.
Let us also remark that the periodic $L_2$-discrepancy is (up to a factor) sometimes also called diaphony. 
This terminology was introduced in \cite{Zin1997}.

\section{Optimal Cubature Points}
In this section we deal with (local) optimality conditions for a set of two-dimensional points $\cP_N \equiv (\bsx, \bsy) \subset \T^2$, where $\bsx, \bsy \in \T^N$ denote the 
vectors of the first and second components of the points, respectively.

\subsection{Optimization Problem}

We want to minimize the squared worst case error 
\begin{align*}
	 {\rm wce}  & ( H^{1,\gamma}_\text{mix}(\T^2) ,\cP_N)^2 =  -1 + \frac{1}{N^2} \sum_{i,j=0}^{N-1}  K_{1,\gamma}(x_i, x_j) \, K_{1,\gamma}(y_i, y_j) 	\\
		= & -1 + \frac{1}{N^2} \sum_{i,j=0}^{N-1}  \left(1 + \gamma k(|x_i - x_j|)  + \gamma k(|y_i - y_j|) + \gamma^2 k(|x_i - x_j|)  k(|y_i - y_j|) \right)	\\
		= & \frac{\gamma}{N^2} \sum_{i,j=0}^{N-1} \left(  k(|x_i - x_j|)  + k(|y_i - y_j|) + \gamma k(|x_i - x_j|)  k(|y_i - y_j|) \right) \\
		= & \frac{\gamma ( 2 k(0) + \gamma k(0)^2) }{N}  \\
		  & \quad + \frac{2 \gamma}{N^2} \sum_{i=0}^{N-2} \sum_{j=i+1}^{N-1} \left( k(|x_i - x_j|)  + k(|y_i - y_j|) + \gamma k(|x_i - x_j|)  k(|y_i - y_j|) \right) \\
\end{align*}
Thus, minimizing ${\rm wce}  ( H^{1,\gamma}_\text{mix}(\T^2) ,\cP_N)^2$ is equivalent to minimizing either
\begin{equation}
	 F_\gamma(\bsx, \bsy) := \sum_{i=0}^{N-2} \sum_{j=i+1}^{N-1} \left( k(|x_i - x_j|)   + k(|y_i - y_j|) + \gamma k(|x_i - x_j|)  k(|y_i - y_j|) \right)
\end{equation}
or 
\begin{equation}
	 G_\gamma(\bsx, \bsy) := \sum_{i,j=0}^{N-1} (1+\gamma k(|x_i - x_j|))(1+\gamma k(|y_i - y_j|) ).
\end{equation}
For theoretical considerations we will sometimes use $G_\gamma$, while for the numerical implementation we will use $F_\gamma$ as objective function, since it has less summands.

Let $\tau,\sigma \in S_N$ be two permutations of $\{0,1,\dots,N-1\}$. Define the sets 
\begin{equation}
	D_{\tau, \sigma} = \left\{\bsx \in [0,1)^N, \bsy \in [0,1)^N:   \begin{matrix}
	    x_{\tau(0)} \leq x_{\tau(1)} \leq \cdots \leq x_{\tau(N-1)} \\ 
	    y_{\sigma(0)} \leq y_{\sigma(1)} \leq \cdots \leq y_{\sigma(N-1)}  
	    \end{matrix}
	\right\}
\end{equation}
on which all points maintain the same order in both components and hence it holds $|x_i - x_j| = s_{i,j} (x_i-x_j)$ for $s_{i,j} \in \{-1,1\}$. It follows that  the restriction of $F_\gamma$ to $D_{\tau, \sigma}$, i.e. $F_\gamma(\bsx, \bsy)_{|D_{\tau,\sigma}}$, is a polynomial of degree $4$ in $(\bsx, \bsy)$. Moreover, $F_{\gamma{|D_{\tau,\sigma}}}$ is convex for sufficiently small $\gamma$.

\begin{proposition}\label{prop_conv}
	$F_\gamma(\bsx, \bsy)_{|D_{\tau,\sigma}}$ and $G_\gamma(\bsx, \bsy)_{|D_{\tau,\sigma}}$ are convex if $\gamma \in [0, 6]$.
\end{proposition}
	
\begin{proof}
  It is enough to prove the claim for $$G_\gamma(\bsx, \bsy) = \sum_{i,j=0}^{N-1} (1+\gamma k(|x_i - x_j|))(1+\gamma k(|y_i - y_j|) ).$$
	Since the sum of convex functions is convex and since $f(x-y)$ is convex if $f$ is, it is enough to show that
	$ f(s,t) =  \big( 1+\gamma k(s) \big) \big( 1+\gamma k(t) \big) $ is convex for $s,t \in [0,1]$.
  To this end, we show that the Hesse matrix $\cH(f)$  is positive definite if $0 \le \gamma<6$.
	First, $f_{ss} = \gamma  \big( 1+\gamma k(t) \big)$ is positive if $\gamma<24$.
	Hence is is enough to check that the determinant of $\cH(f)$ is positive, which is equivalent to the inequality
	$$ \big( 1+\gamma k(s) \big) \big( 1+\gamma k(t) \big)  > \gamma^2 \left( s-\frac12 \right)^2 \left( t-\frac12 \right)^2.$$
	So it remains to see that $$  1+\gamma k(s) =  1 + \frac{\gamma}{2} \left( s^2-s+\frac16 \right) > \gamma \left( s-\frac12 \right)^2.$$
	But this is elementary to check for $0\le \gamma < 6$ and $s\in [0,1]$.
	In the case $\gamma=6$ the determinant of $\cH(f)=0$ and some additional argument is necessary which we omit here.
\qed\end{proof}
Since
\[
	[0,1)^N \times [0,1)^N = \bigcup_{(\tau, \sigma) \in S_N \times S_N} D_{\tau, \sigma} ,
\]
one can obtain the global minimum of $F_\gamma$ on $[0,1)^N \times [0,1)^N$ by  computing $\argmin_{(\bsx, \bsy) \in D_{\tau,\sigma}} F_\gamma(\bsx, \bsy)$ for all $(\tau, \sigma) \in S_N \times S_N$ and choose the global minimum as the smallest of all the local ones.



\subsection{Using the Torus Symmetries}

We now want to analyze how symmetries of the two dimensional torus $\mathbb{T}^2$ allow to reduce the number of regions $D_{\tau,\sigma}$ for which the optimization problem has to be solved.

The symmetries of the torus $\mathbb{T}^2$ which do not change the worst case error for the considered classes of periodic functions are generated by
\begin{enumerate}
	\item Shifts in the first coordinate $x \mapsto x + c \mod 1$ and shifts in the second coordinate $y \mapsto y + c \mod 1$.
	\item Reflection of the first coordinate $x \mapsto 1-x$ and reflection of the second coordinate $y \mapsto 1-y$.
	\item Interchanging the first coordinate $x$ and the second coordinate $y$.
	\item The points are indistinguishable, hence relabeling the points does not change the worst case error.
\end{enumerate}
Applying finite compositions of these symmetries to all the points in the point set $\cP_N = \{ (x_0,y_0), \dots , (x_{N-1},y_{N-1}) \}$ leads to an equivalent point set with the same
worst case integration error.
This shows that the group of symmetries $G$ acting on the pairs $(\tau,\sigma)$ indexing $D_{\tau,\sigma}$ generated by the following operations
\begin{enumerate}
	\item replacing $\tau$ or $\sigma$ by a shifted permutation: $\tau \mapsto ( \tau(0) + k \mod N , \dots, \tau(N-1) + k \mod N )$ or $\sigma \mapsto ( \sigma(0) + k \mod N , \dots, \sigma(N-1) + k \mod N )$
	\item replacing $\tau$ or $\sigma$ by its flipped permutation:  $\tau \mapsto ( \tau(N-1), \tau(N-2), \dots, \tau(1), \tau(0) )$ or $\sigma \mapsto ( \sigma(N-1), \sigma(N-2), \dots, \sigma(1), \sigma(0) )$
	\item interchanging $\sigma$ and $\tau$: $(\tau,\sigma) \mapsto ( \sigma,\tau)$
	\item applying a permutation $\pi\in S_N$ to both $\tau$ and $\sigma$ : $(\tau,\sigma) \mapsto (\pi \tau, \pi \sigma)$
\end{enumerate}
lead to equivalent optimization problems. So let us call the pairs $(\tau,\sigma)$ and $(\tau',\sigma')$ in $S_N \times S_N$ equivalent if they are in the same orbit with respect to the action of $G$.
In this case we  write $(\tau,\sigma) \sim (\tau',\sigma')$.

Using the torus symmetries 1. and 4. it can always be arranged that $\tau = \text{id}$ and $\sigma(0)=0$, which together with fixing the point $(x_0,y_0)=(0,0)$ leads to the
sets 
\begin{equation}
	D_{\sigma} = \left\{\bsx \in [0,1)^N, \bsy \in [0,1)^N:   \begin{array}{rl}
	    0= & x_0 \leq x_1 \leq \ldots \leq x_{N-1} \\ 
	    0= & y_{0} \leq y_{\sigma(1)} \leq \cdots \leq y_{\sigma(N-1)}
	    \end{array}
	\right\},
\end{equation}
where $\sigma \in S_{N-1}$ denotes a permutation of $\{1,2,\ldots, N-1\}$.

But there are many more symmetries and it would be algorithmically desirable to cycle through exactly one representative of each equivalence class
without ever touching the other equivalent $\sigma$. This seems to be difficult to implement, hence we settled for a little less which still reduces the amount of permutations to be handled considerably.

To this end, let us define the symmetrized metric 
\begin{equation}
 d(i,j) = \min \{ |i-j| , N - |i-j| \} \qquad \mbox{for}\qquad 0 \le i,j \le N-1
\end{equation}
and the following subset of $S_{N}$.
\begin{definition}
	The set of \emph{semi-canonical} permutations $\mathfrak{C}_N \subset S_N$ consists of permutations $\sigma$ which fulfill
	\begin{enumerate}
		\item[(i)] $\sigma(0)=0$
		\item[(ii)] $d(\sigma(1),\sigma(2)) \le d(0,\sigma(N-1))$
		\item[(iii)] $ \sigma(1) = \min \left\{ d(\sigma(i),\sigma(i+1)) \mid i=0,1,\dots,N-1 \right\}$
		\item[(iv)] $\sigma$ is lexicographically smaller than $\sigma^{-1}$.
	\end{enumerate}
	Here we identify $\sigma(N)$ with $0=\sigma(0)$.
\end{definition}
%
%
This means that $\sigma$ is {\em semi-canonical} if the distance between $0=\sigma(0)$ and $\sigma(1)$ is minimal among all distances between $\sigma(i)$ and $\sigma(i+1)$,
which can be arranged by a shift. Moreover, the distance between $\sigma(1)$ and $\sigma(2)$ is at most as large as the distance between $\sigma(0)$ and $\sigma(N-1)$,
which can be arranged by a reflection and a shift if it is not the case.
Hence we have obtained the following lemma. 
\begin{lemma}
  For any permutation $\sigma \in S_N$ with $\sigma(0)=0$ there exists a semi-canonical $\sigma'$ such that the sets $D_\sigma$ and $D_{\sigma'}$ are equivalent up to torus symmetry.
\end{lemma}
Thus we need to consider only semi-canonical $\sigma$ which is easy to do algorithmically. 

\begin{remark}
	If $\sigma \in S_{N}$ is semi-canonical, it holds $\sigma(1) \leq N/2$.
\end{remark}
Another main advantage in considering our objective function only in domains $D_\sigma$ is that it is not only convex but strictly convex here. This is due
to the fact that we fix $(x_0,y_0)=(0,0)$. 

\begin{proposition} \label{prop_convexity}
	$F_\gamma(\bsx, \bsy)_{|D_{\sigma}}$ and $G_\gamma(\bsx, \bsy)_{|D_{\sigma}}$ are strictly convex if $\gamma \in [0, 6]$.
\end{proposition}

\begin{proof}
  Again it is enough to prove the claim for $$G_\gamma(\bsx, \bsy) = \sum_{i,j=0}^{N-1} (1+\gamma k(|x_i - x_j|))(1+\gamma k(|y_i - y_j|) ).$$
	Now we use that the sum of a convex and a strictly convex function is again strictly convex.
	Hence it is enough to show that the function
	\begin{align*}
	  f(x_1,\dots,x_{N-1},y_1,\dots,y_{N-1} ) & = \sum_{i=1}^{N-1} (1+\gamma k(|x_i - x_0|))(1+\gamma k(|y_i - y_0|) ) \\
		                                        & = \sum_{i=1}^{N-1} (1+\gamma k(x_i))(1+\gamma k(y_i) ) 
	\end{align*}
	is strictly convex on $[0,1]^{N-1} \times [0,1]^{N-1}$.
	In the proof of Proposition \ref{prop_conv} it was actually shown that $f_i(x_i,y_i)=(1+\gamma k(x_i))(1+\gamma k(y_i) )$ is strictly convex for
	$(x_i,y_i) \in [0,1]^2$ for each fixed $i=1,\dots,N-1$. 
	Hence the strict convexity of $f$ follows from the following easily verified lemma. \qed
	\begin{lemma}
	  Let $f_i:D_i \to \R, i=1,\dots,m$ be strictly convex functions on the convex domains $D_i \in \R^{d_i}$.
		Then the function $$ f: D= D_1 \times \dots\times D_m  \to \R, (z_1,\dots,z_m) \mapsto \sum_{i=1}^m f_i(z_i)$$ is strictly convex. 
	\end{lemma}
\end{proof}
Hence we have indeed a unique point in each $D_\sigma$ where the minimum of $F_\gamma$ is attained.

\subsection{Minimizing $F_\gamma$ on $D_\sigma$}

Our strategy will be to compute the local minimum of $F_\gamma$ on each region \linebreak $D_\sigma \subset [0,1)^N \times [0,1)^N$ for all semi-canonical permutations $\sigma \in \mathfrak{C}_N \subset S_{N}$ and determine the  global minimum by choosing the smallest of all the local ones. 

This gives for each $\sigma \in \mathfrak{C}_N$ the constrained optimization problem
\begin{equation} \label{eqn_constr_mini}
 \min_{(\bsx, \bsy)  \in D_\sigma} F_\gamma(\bsx, \bsy) \quad \text{ subject to } v_i(\bsx) \geq 0 \text { and } w_i(\bsy) \geq 0 \text{ for all } i=1,\ldots, N-1 ,
\end{equation}
where the inequality constraints are linear and given by
\begin{equation}
	v_i(\bsx) = x_{i} - x_{i-1} \quad \text{ and } \quad w_i(\bsy) = y_{\sigma(i)} - y_{\sigma(i-1)} \quad \text{ for } i=1,\ldots, N-1 .
\end{equation}
%
%
In order to use the necessary (and due to local strict convexity also sufficient) conditions for local minima
\begin{align*}
	\frac{\partial}{\partial x_k} F_\gamma(\bsx, \bsy) =  0 \quad \text{ and } \quad \frac{\partial}{\partial y_k} F_\gamma(\bsx, \bsy) =  0 \quad \text{ for } k=1,\ldots, N-1
\end{align*}
for $(\bsx, \bsy) \in D_\sigma$  we need to evaluate the partial derivatives of $F_\gamma$. 

\begin{proposition}
 For a given permutation $\sigma \in \mathfrak{C}_N$ the partial derivative of $F_{\gamma|D_\sigma}$ with respect to the second component $\bsy$ is given by 
 \begin{equation} \label{eqn_first_deriv}
  \frac{\partial}{\partial y_k} F_\gamma(\bsx, \bsy)_{|D_\sigma} = y_k \left( \sum_{\substack{ i=0 \\ i \neq k} }^{N-1}  c_{i,k}  \right)   - \sum_{\substack{ i=0 \\ i \neq k} }^{N-1} c_{i,k}  y_i + \frac{1}{2} \left( \sum_{i=0}^{k-1} c_{i,k} s_{i,k} - \sum_{j=k+1}^{N-1} c_{k,j} s_{k,j} \right),
 \end{equation}
where $s_{i,j} = \text{sgn}(y_i - y_j)$ and $c_{i,j} := 1+ \gamma  k(|x_i - x_j|) = c_{j,i}$.

Interchanging $\bsx$ and $\bsy$ the same result holds for the partial derivatives with respect to $\bsx$ with the obvious modification to $c_{i,j}$ and the simplification that $s_{i,j} = -1$.

The second order derivatives with respect to $\bsy$  are given by
\begin{equation} \label{eqn_second_deriv}
	\frac{\partial^2}{\partial y_k \partial y_j} F(\bsx, \bsy)_{|D_\sigma} = 	\begin{cases}
													\sum_{i=0 }^{k-1}  c_{i,k} + \sum_{i=k+1 }^{N-1}  c_{i,k} & \text{ for } j=k \\
													-c_{k,j}	& \text{ for } j \neq k
												\end{cases}, \quad k,j \in \{1,\ldots, N-1\}
\end{equation}
Again, the analogue for $\frac{\partial^2}{\partial x_k \partial x_j} F(\bsx, \bsy)_{|D_\sigma}$ is obtained with the obvious modification $c_{i,j} = 1+\gamma k(|y_i-y_j|)$.
\end{proposition}

\begin{proof}
We prove the claim for the partial derivative with respect to $\bsy$:
\begin{align*}
	\frac{\partial}{\partial y_k} F_\gamma(\bsx, \bsy) = &  \sum_{i=0}^{N-2} \sum_{j=i+1}^{N-1}  \frac{\partial}{\partial y_k} k(|y_i - y_j|)   \underbrace{ \left( 1+ \gamma  k(|x_i - x_j|) \right) }_{ =: c_{i,j} } + \frac{\partial}{\partial y_k} k(|x_i - x_j|)  \\ 
	= &  \sum_{i=0}^{N-2} \sum_{j=i+1}^{N-1}    c_{i,j} \,   \frac{\partial}{\partial y_k} k(|y_i - y_j|)   \\ 
	= & \sum_{i=0}^{N-2} \sum_{j=i+1}^{N-1}    c_{i,j}\;   k'( s_{i,j} \,  (y_i - y_j )) \cdot \begin{cases}
		 														s_{i,j} & \text{ for } i = k \\
																-s_{i,j} & \text{ for } j = k \\
																0 & \text{ else  }
															     \end{cases} \\ 
	= & \sum_{j=k+1}^{N-1} c_{k,j}   s_{k,j} \;   \left( s_{k,j} \,  (y_k - y_j ) - \frac{1}{2} \right)  -  \sum_{i=0}^{k-1} c_{i,k}   s_{i,k} \;  \left( s_{i,k} \,  (y_i - y_k ) - \frac{1}{2} \right) \\
	= & y_k \left( \sum_{\substack{ i=0 \\ i \neq k}}^{N-1}  c_{i,k}  \right)   -  \sum_{\substack{ i=0 \\ i \neq k}}^{N-1} c_{i,k}  y_i + \frac{1}{2} \left( \sum_{i=0}^{k-1} c_{i,k} s_{i,k} - \sum_{j=k+1}^{N-1} c_{k,j} s_{k,j} \right) .
\end{align*}
From this we immediately get the second derivative \eqref{eqn_second_deriv}. \qed
\end{proof}


\subsection{Lower Bounds of $F_\gamma$ on $D_\sigma$}
Until now we are capable of approximating local minima of $F_\gamma$ on a given $D_\sigma$. If this is done for all $\sigma \in \mathfrak{C}_N$ we can obtain a candidate for a global minimum, but  due to the finite precision of floating point arithmetic one can never be sure to be close to the actual global minimum.
However, it is also possible to compute a lower bound for the optimal point set for each $D_\sigma$ using Wolfe-duality for constrained optimization.
It is known \cite{NW06} that for a convex problem with linear inequality constraints like \eqref{eqn_constr_mini}  the Lagrangian
\begin{align}
	\cL_F(\bsx, \bsy, \bslambda, \bsmu) := & F(\bsx, \bsy) - \bslambda^T \bsv(\bsx) - \bsmu^T \bsw(\bsy) \\
	= & F(\bsx, \bsy) - \sum_{i=1}^{N-1} \left( \lambda_i v_i(\bsx) + \mu_i w_{i}(\bsy) \right)
\end{align}
gives a lower bound on $F$, i.e. 
$$\min_{(\bsx, \bsy) \in D_\sigma} F(\bsx, \bsy) \geq \cL_F(\tilde{\bsx}, \tilde{\bsy	}, \bslambda, \bsmu) $$
    for all $(\tilde{\bsx}, \tilde{\bsy}, \bslambda, \bsmu)$ that fulfill the constraint 
	\begin{equation}
		\nabla_{(\bsx, \bsy)}  \cL_F(\tilde{\bsx}, \tilde{\bsy}, \bslambda, \bsmu) = 0 \quad \text{ and } \quad \bslambda, \bsmu \geq 0 \text{ (component-wise)}.
	\end{equation}
Here, $\nabla_{(\bsx, \bsy)} = (\nabla_{\bsx}, \nabla_{\bsy})$, where $\nabla_{\bsx}$ denotes the gradient of a function with respect to the variables in $\bsx$.
Hence it is our goal to find for each $D_\sigma$ such an admissible point $(\tilde{\bsx}, \tilde{\bsy}, \bslambda, \bsmu)$ which yields a lower bound that is larger than 
some given candidate for the global minimum. If the relevant computations are carried out in infinite precision rational number arithmetic these bounds are mathematically reliable.

In order to accomplish this we first have to compute the Lagrangian of \eqref{eqn_constr_mini}. To this end, let  ${\bf P}_\sigma \in \{-1,0,1\}^{(N-1) \times (N-1)}$ denote the permutation matrix corresponding to $\sigma \in S_{N-1}$ and 
\begin{equation} \label{eqnB} 
	\bsB := \begin{pmatrix}
	        	1 & -1 & 0 & \ldots & 0 & 0 \\
	        	0 & 1 & -1 & \ldots & 0 & 0 \\
	        	\vdots & & & \ddots & & \vdots \\
	        	0 & & \ldots & 0 & 1 & -1  \\
	        	0 & & \ldots &  &0  & 1  \\
	        \end{pmatrix} \in \R^{(N-1) \times (N-1)} .
\end{equation}
Then the partial derivatives of $\cL_F$ with respect to $\bsx$ and $\bsy$ are given by
\begin{align}
	\nabla_{\bsx} \cL_F(\bsx, \bsy, \bslambda, \bsmu) = & \nabla_{\bsx} F(\bsx, \bsy) - \begin{pmatrix}
	                                                                                    	\lambda_1 - \lambda_2 \\
	                                                                                    	\vdots \\
	                                                                                    	\lambda_{N-2} - \lambda_{N-1} \\
	                                                                                    	\lambda_{N-1}
	                                                                                    \end{pmatrix}
	                                                                                    = \nabla_{\bsx} F(\bsx, \bsy) - \bsB \bslambda
\end{align}
and
\begin{align}
	\nabla_{\bsy} \cL_F(\bsx, \bsy, \bslambda, \bsmu) = & \nabla_{\bsy} F(\bsx, \bsy) - \begin{pmatrix}
	                                                                                    	\mu_{\sigma(1)} - \mu_{\sigma(2)} \\
	                                                                                    	\vdots \\
	                                                                                    	\mu_{\sigma(N-2)} - \mu_{\sigma(N-1)} \\
	                                                                                    	\mu_{\sigma(N-1)}
	                                                                                    \end{pmatrix}
	                                                                                    = \nabla_{\bsy} F(\bsx, \bsy) -  \bsB {\bf P}_\sigma \bsmu .
\end{align}
This leads to the following theorem.
\begin{theorem} \label{thm_lower_bound}
	For $\sigma \in \mathfrak{C}_N$ and $\delta > 0$ let the point $(\tilde{\bsx}_\sigma, \tilde{\bsy}_\sigma) \in D_\sigma$ fulfill
	\begin{equation} \label{eqn_offset_grad}
		\frac{\partial}{\partial x_k} F(\tilde{\bsx}_\sigma, \tilde{\bsy}_\sigma) = \delta \quad \text{ and } \quad\frac{\partial}{\partial y_k} F(\tilde{\bsx}_\sigma, \tilde{\bsy}_\sigma) = \delta \quad \text{ for } k=1,\ldots, N-1 .
	\end{equation}
	Then 
	\begin{align} \label{eqn_LowerBound}
		F(\bsx, \bsy) \geq & F(\tilde{\bsx}_\sigma, \tilde{\bsy}_\sigma) - \delta \sum_{i=1}^{N-1} \left( (N-i) \cdot v_i(\tilde{\bsx}_\sigma) + \sigma(N-i) w_i(\tilde{\bsy}_\sigma) \right)  \\
			> & F(\tilde{\bsx}_\sigma, \tilde{\bsy}_\sigma) - \delta N^2 \label{eqn_LowerBound2}
	\end{align}
	holds for all $(\bsx, \bsy) \in D_\sigma$.
\end{theorem}
\begin{proof}
Choosing 
\begin{equation} \label{eqn_multipl}
	\bslambda = \bsB^{-1} \nabla_{\bsx} F(\tilde{\bsx}_\sigma, \tilde{\bsy}_\sigma) \quad \text{ and } \quad \bsmu = {\bf P}_\sigma^{-1}  \bsB^{-1}  \nabla_{\bsy} F(\tilde{\bsx}_\sigma, \tilde{\bsy}_\sigma) 
\end{equation}
yields 
\begin{equation} 
	\nabla_{\bsx} F(\tilde{\bsx}, \tilde{\bsy}) = \bsB \bslambda \quad \text{ and } \quad \nabla_{\bsy} F(\tilde{\bsx}, \tilde{\bsy}) =  \bsB {\bf P}_\sigma \bsmu .
\end{equation}
A short computation shows that the inverse of $\bsB$ from \eqref{eqnB} is given by
\[
	\bsB^{-1} := \begin{pmatrix}
	        	1 & 1 &  \ldots & 1 \\
	        	0 & 1 & \ldots & 1 \\
	        	\vdots & 0 & \ddots & \vdots \\
	        	0 & \ldots & 0 & 1  \\
	        \end{pmatrix} \in \R^{(N-1) \times (N-1)} ,
\]
which yields $\bsy, \bslambda >0$ and hence by Wolfe duality gives \eqref{eqn_LowerBound}. The second inequality \eqref{eqn_LowerBound2} then follows from noting that both $|v_i(\bsx)|$ and $|w_i(\bsy)|$ are bounded by $1$ and $2 \sum_{i=1}^{N-1} \sigma(N-i) =  2\sum_{i=1}^{N-1} i$ = $(N-1)(N-2) < N^2$. \qed
\end{proof}
Now, suppose we had some candidate $(\bsx^*, \bsy^*)\in D_{\sigma^*}$ for an optimal point set. If we can find for all other $\sigma \in \mathfrak{C}_N$ points $(\tilde{\bsx}_\sigma, \tilde{\bsy}_\sigma)$ that fulfills \eqref{eqn_offset_grad} and  
$$F(\tilde{\bsx}_\sigma, \tilde{\bsy}_\sigma) - \delta N^2 \geq F_\gamma(\bsx^*, \bsy^*)$$
for some $\delta > 0$, we can be sure that $D_{\sigma^*}$ is (up to torus symmetry) the unique domain $D_\sigma$ that contains the globally optimal point set.

\section{Numerical Investigation of Optimal point sets}
In this section we numerically obtain optimal point sets with respect to the worst case error in $H^1_\text{mix}$. Moreover, we present a proof by exhaustion that these point sets are indeed approximations to the unique  (modulo torus symmetry) minimizers of $F_\gamma$. 
Since integration lattices are local minima,  if the $D_\sigma$ containing the global minimizer corresponds to an integration lattice, this integration lattice is the exact global minimizer.

\subsection{Numerical Minimization with Alternating Directions}
In order to obtain the global minimum $(\bsx^*, \bsy^*)$ of $F_\gamma$ we are going to compute
\begin{equation}
	\sigma^* := \argmin_{\sigma \in \mathfrak{C}_N} \min_{(\bsx, \bsy) \in D_\sigma} F_\gamma(\bsx, \bsy),
\end{equation}
where the inner minimum has a unique solution due to Proposition \ref{prop_convexity}.  Moreover, since $D_\sigma$ is a convex domain we know that the local minimum of $F_\gamma(\bsx, \bsy)_{|D_{\sigma}}$ is not on the boundary. Hence we can restrict our search for optimal point sets to the interior of $D_\sigma$, where $F_\gamma$ is differentiable. 

Instead of directly employing a local optimization technique, we will make use of the special structure of $F_\gamma$.
While $F_\gamma(\bsx, \bsy)_{|D_{\sigma}}$ is a polynomial of degree four, the functions
\begin{equation} \label{eqn_quadpoly}
 \bsx \mapsto F_\gamma(\bsx, \bsy_0)_{|D_{\sigma}} \quad \text{ and } \quad \bsy \mapsto F_\gamma(\bsx_0, \bsy)_{|D_{\sigma}},
\end{equation}
where one coordinate direction is fixed, are quadratic polynomials, which have unique minima in $D_\sigma$. We are going to use this property within an alternating minimization approach. This means, that the objective function $F$ is not minimized along all coordinate directions simultaneously, but  with respect to certain successively  alternating blocks of coordinates. If these blocks have size one this method is usually referred to as \emph{coordinate descent} \cite{LuoTseng} or nonlinear Gau\ss -Seidel method \cite{Grippo2000}. It is successfully employed in various applications, like e.g. expectation maximization or tensor approximation \cite{McLachlan, Uschmajew}.

In our case we will alternate between minimizing $F_\gamma(\bsx, \bsy)$ along the first coordinate block $\bsx \in (0,1)^{N-1}$ and the second one $\bsy \in (0,1)^{N-1}$, which can be done exactly due to the quadratic polynomial property of the partial objectives \eqref{eqn_quadpoly}. The method is outlined in Algorithm \ref{algo_ao}, which for threshold-parameter $\delta=0$ approximates the local minimum of
 $F_\gamma$ on $D_\sigma$. For $\delta>0$ it obtains feasible points that fulfill \eqref{eqn_offset_grad}, i.e. $\nabla_{(\bsx, \bsy)} F_\gamma = (\delta, \ldots, \delta)= \delta {\bf 1}$.
Linear convergence of the alternating optimization method for strictly convex functions was for example proven in \cite{Ortega, bezdek87}.

\begin{algorithm}[tb] 
    \textbf{Given:} Permutation $\sigma \in \mathfrak{C}_N$, tolerance $\varepsilon > 0$ and off-set $\delta \geq 0$. \\
   \textbf{Initialize:} 
    \begin{enumerate}
	    \item $\bsx^{(0)} := (0, \frac{1}{N}, \ldots, \frac{N-1}{N})$ and $\bsy^{(0)} = (0, \frac{\sigma(1)}{N}, \ldots, \frac{\sigma(N-1)}{N})$.
	    \item $k:=0$.
    \end{enumerate}
    
    \Repeat{ $ \sqrt{ \|\nabla_{\bsx}\|^2 + \|\nabla_{\bsy}\|^2 } < \varepsilon$ }{
	  \vspace{-0.3cm}
	\begin{enumerate}
	    \item compute $\bsH_{\bsx} := \left( \partial_{x_i} \partial_{x_j} F_\gamma(\bsx^{(k)}, \bsy^{(k)} \right)_{i,j=1}^N$ and $\nabla_{\bsx} = \left(  \partial_{x_i} F_\gamma(\bsx^{(k)}, \bsy^{(k)} \right)_{i=1}^N$ by \eqref{eqn_second_deriv} and \eqref{eqn_first_deriv}.
	    \item Update $\bsx^{(k+1)} := \bsH_{\bsx}^{-1} \left( \nabla_{\bsx} + \delta \mathbf{1} \right) $ via Cholesky factorization.
	    \item compute $\bsH_{\bsy} := \left( \partial_{y_i} \partial_{y_j} F_\gamma(\bsx^{(k+1)}, \bsy^{(k)} \right)_{i,j=1}^N$ and $\nabla_{\bsy} = \left(  \partial_{y_i} F_\gamma(\bsx^{(k+1)}, \bsy^{(k)} \right)_{i=1}^N$.
	    \item Update $\bsy^{(k+1)} := \bsH_{\bsy}^{-1} \left( \nabla_{\bsy} + \delta \mathbf{1} \right)$  via Cholesky factorization.
	    \item $k := k+1$.
	\end{enumerate}
    }
    \textbf{Output}: point set $(\bsx, \bsy) \in D_\sigma$ with  $\nabla_{\bsx} F_\gamma(\bsx, \bsy) \approx \delta \mathbf{1}$ and $\nabla_{\bsy} F_\gamma(\bsx, \bsy) \approx \delta \mathbf{1}$.
\caption{Alternating minimization algorithm. For off-set $\delta=0$ it finds local minima of $F_\gamma$. For $\delta > 0$ it obtains feasible points used by Algorithm \ref{algo_lowerbound}. } \label{algo_ao}
\end{algorithm}

\subsection{Obtaining Lower Bounds}
By now we are able to obtain a point set $(\bsx^*, \bsy^*) \in D_{\sigma^*}$ as a candidate for a global minimum of $F_\gamma$ by finding local minima on each $D_\sigma, \sigma \in \mathfrak{C}_N$. On first sight we can not be sure that we chose the right $\sigma^*$, because the value of $\min_{(\bsx, \bsy) \in D_\sigma} F_\gamma(\bsx, \bsy)$ can only be computed numerically. 

On the other hand, Theorem \ref{thm_lower_bound} allows to compute lower bounds for all the other domains $D_\sigma$ with $\sigma \in \mathfrak{C}_N$. If we were able to obtain for each $\sigma$ a point $(\tilde{\bsx}_\sigma, \tilde{\bsy}_\sigma)$, such that 
$$\min_{(\bsx,\bsy) \in D_{\sigma^*}} F_\gamma(\bsx,\bsy) \approx \theta_N := F_\gamma(\bsx^*,\bsy^*) < \cL_F(\tilde{\bsx}_\sigma, \tilde{\bsy}_\sigma) - 2N^2 \delta \leq F_\gamma(\bsx, \bsy),$$
we could be sure that the global optimum  is indeed located in $D_{\sigma^*}$ and $(\bsx^*, \bsy^*)$ is a good approximation to it. Luckily, this is the case. Of course  certain computations can not be done in standard double floating point arithmetic. Instead we use arbitrary precision rational number  (APR) arithmetic from the GNU Multiprecision library GMP from \url{http://www.gmplib.org}. Compared to standard floating point arithmetic in double precision this is very expensive, but it has only to be used at certain parts of the algorithm.
The resulting procedure is outlined in Algorithm \ref{algo_lowerbound}, where we marked those parts which require APR arithmetic.

\begin{algorithm}[tb] 

   \textbf{Given:} Optimal point candidate $\cP_N:= (\bsx^*, \bsy^*) \in D_\sigma$ with $\sigma \in \mathfrak{C}_N$, tolerance $\varepsilon > 0$ and off-set $\theta \geq 0$. \\
   \textbf{Initialize:} 
    \begin{enumerate}
	    \item Compute $\theta_N := F_\gamma(\bsx^*, \bsy^*)$ \begin{color}{red} (in APR arithmetic) \end{color}.
	    \item $\Xi_N := \emptyset$.
    \end{enumerate}
    
    \For{ \textbf{\emph{all}} \,\,  $\sigma \in \mathfrak{C}_N$ }{
	\begin{enumerate}
	    \item Find $(\tilde{\bsx}_\sigma, \tilde{\bsy}_\sigma) \in D_\sigma$ s.t. $\nabla_{(\bsx, \bsy)}F_\gamma(\tilde{\bsx}_\sigma, \tilde{\bsy}_\sigma) \approx \delta \mathbf{1}$ by Algorithm \ref{algo_ao}.
	    \item Compute $\bslambda := \bsB^{-1} \nabla_{\bsx} F(\tilde{\bsx}_\sigma, \tilde{\bsy}_\sigma) \quad \text{ and } \quad \bsmu := {\bf P}_\sigma^{-1}  \bsB^{-1}  \nabla_{\bsy} F(\tilde{\bsx}_\sigma, \tilde{\bsy}_\sigma)$ \begin{color}{red} (in APR arithmetic) \end{color}. 
	    \item Verify $\bslambda, \bsmu > 0$.
	    \item Evaluate $\beta_\sigma := \cL_{F_\gamma}(\tilde{\bsx}_\sigma, \tilde{\bsy}_\sigma, \bslambda, \bsmu)$ \begin{color}{red} (in APR arithmetic) \end{color}.
	    \item \textbf{If} ( $\beta_\sigma \leq \theta_N$ )  $\Xi_N := \Xi_N \cup \sigma$.
	\end{enumerate}
    }
    \textbf{Output}: Set $\Xi$ of permutations $\sigma$ in which $D_\sigma$ contained a lower bound smaller than $\theta_N$.
\caption{Computation of lower bound on $D_\sigma$.} \label{algo_lowerbound}
\end{algorithm}

\subsection{Results}

In Figures \ref{fig_plots1} and \ref{fig_plots2} the optimal point sets for $N=2,\ldots, 16$ and both $\gamma =1$ and $\gamma =6$ are plotted. It can be seen that they are close to lattice point sets, which justifies using them as start points in Algorithm \ref{algo_ao}. The distance to lattice points seems to be small if $\gamma$ is small.

In Table \ref{table_results} we list the permutations $\sigma$ for which $D_\sigma$ contains an optimal set of cubature points. In the second column the total number of semi-canonical permutations $\mathfrak{C}_N$ that had to be considered is shown. It grows approximately like $\frac{1}{2}(N-2)!$. Moreover, we computed the minimal worst case error and periodic $L_2$-discrepancies.

In some cases we found more than one semi-canonical permutation $\sigma$ for which $D_\sigma$ contained a point set which yields the optimal worst case error.  Nevertheless, they represent equivalent permutations.
In the following list, the torus symmetries used to show the equivalency of the permutations are given. All operations are modulo 1.
\begin{itemize}
 \item $N=7$:  $(x,y) \mapsto (1-y,x)$ 
 \item $N=9$:  $(x,y) \mapsto (y-2/9, x-1/9)$
 \item $N=11$: $(x,y) \mapsto (y+5/11 , x-4/11 )$
 \item $N=14$: $(x,y) \mapsto (x-4/14 , y+6/14 )$
 \item $N=15$: $(x,y) \mapsto (y+3/15 , x+2/15 ), (y-2/15 , 12/15-x ), (y-6/15 , 4/15-x ) $
 \item $N=16$: $(x,y) \mapsto (1/16-x , 3/16-y )$
\end{itemize}
In all the examined cases $N \in \{2,\ldots, 16\}$ Algorithm 2 produced sets $\Xi_N$ which contained exactly the permutations that were previously obtained by Algorithm \ref{algo_ao} and are listed in Table 1. Thus we  can be sure, that the respective $D_\sigma$ contained minimizers of $F_\gamma$, which on each $D_\sigma$ are unique. Hence we know that our numerical approximation of the minimum is close to the true global minimum, which (modulo torus symmetries)  is unique. In the cases $N=1,2,3,5,7,8,12,13$ the obtained global minima are integration lattices.

\begin{table}[t]
	\begin{tabular}{r|c|c|c|l|c}
	    $N$ & $| \mathfrak{C}_N |$ & ${\rm wce}(H^{1,1}_\text{mix},\cP_N^*)$ & $D_2 ({\cP}_N^*)$ & $\sigma^*$ & Lattice \\
	    \hline
	    \bf 1 & 0 & 0.416667 & 0.372678 & (0) & \checkmark \\
	    \hline
	    \bf 2 & 1 & 0.214492 & 0.212459 & (0 1) & \checkmark \\
	    \hline
	    \bf 3 & 1 & 0.146109 & 0.153826 & (0 1 2) & \checkmark \\
	    \hline
	    4 & 2 & 0.111307 & 0.121181 & (0 1 3 2) &  \\
	    \hline
	    \bf 5 & 5 & 0.0892064 & 0.0980249 & (0 2 4 1 3) & \checkmark \\
	    \hline
	    6 & 13 & 0.0752924 & 0.0850795 & (0 2 4 1 5 3) &  \\
	    \hline
	    7 & 57 & 0.0650941 & 0.0749072 & (0 2 4 6 1 3 5), (0 3 6 2 5 1 4) &  \checkmark \\
	    \hline
	    \bf 8 & 282 & 0.056846 & 0.0651562 & (0 3 6 1 4 7 2 5) & \checkmark \\
	    \hline
	    9 & 1,862 & 0.0512711 & 0.0601654 & (0 2 6 3 8 5 1 7 4), (0 2 7 4 1 6 3 8 5) &  \\
	    \hline
	    10 & 14,076 & 0.0461857 & 0.054473 & (0 3 7 1 4 9 6 2 8 5) &  \\
	    \hline
	    11 & 124,995 &  0.0422449 & 0.050152 & \begin{minipage}{0.4\linewidth}(0 3 8 1 6 10 4 7 2 9 5),\\ (0 3 9 5 1 7 10 4 8 2 6) \end{minipage} &  \\
	    \hline
	    12 & 1,227,562 & 0.0370732 & 0.0456259 & (0 5 10 3 8 1 6 11 4 9 2 7) &  \checkmark \\
	    \hline
	    \bf 13 & 13,481,042 & 0.0355885 & 0.0421763 & (0 5 10 2 7 12 4 9 1 6 11 3 8) & \checkmark \\
	    \hline
	    14 & 160,456,465 & 0.0333232 & 0.0400524 & \begin{minipage}{0.4\linewidth}(0 5 10 2 8 13 4 11 6 1 9 3 12 7),\\  (0 5 10 3 12 7 1 9 4 13 6 11 2 8) \end{minipage}  &  \\
	    \hline
	    15 & 2,086,626,584 & 0.0312562 & 0.0379055 & \begin{minipage}{0.4\linewidth}(0 4 9 13 6 1 11 3 8 14 5 10 2 12 7),\\  (0 5 11 2 7 14 9 3 12 6 1 10 4 13 8),\\ (0 5 11 2 8 13 4 10 1 6 14 9 3 12 7),\\ (0 5 11 2 8 13 6 1 10 4 14 7 12 3 9) \end{minipage} &  \\
	    \hline
	    16 & 29,067,602,676 & 0.0294507 & 0.0359673 & \begin{minipage}{0.4\linewidth}(0 3 11 5 14 9 1 7 12 4 15 10 2 6 13 8),\\  (0 3 11 6 13 1 9 4 15 7 12 2 10 5 14 8)\end{minipage} &  \\
	    \hline
	\end{tabular}
	\caption{List of semi-canonical permutations $\sigma$, such that $D_\sigma$ contains an optimal set of cubature points for $N=1,\ldots, 16$. } \label{table_results}
\end{table}

\begin{figure}[t]
	\begin{tabular}{ccc}
		\includegraphics[width=0.33\linewidth, height=0.32\linewidth]{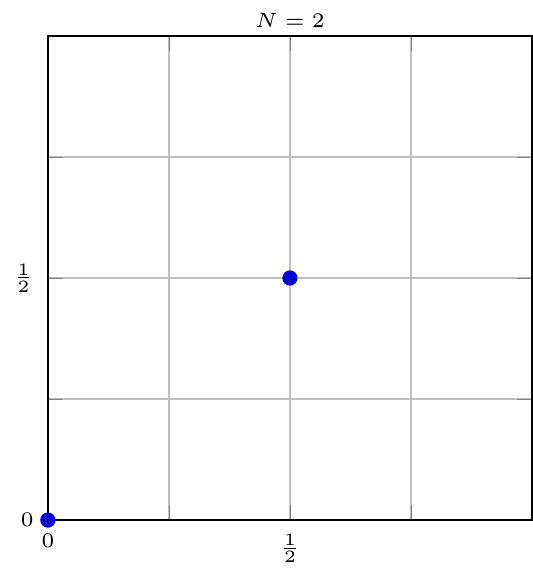} & \includegraphics[width=0.33\linewidth, height=0.32\linewidth]{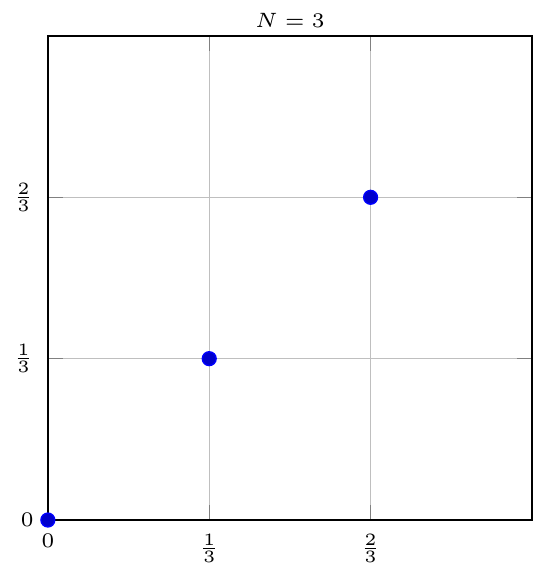} & \includegraphics[width=0.33\linewidth, height=0.32\linewidth]{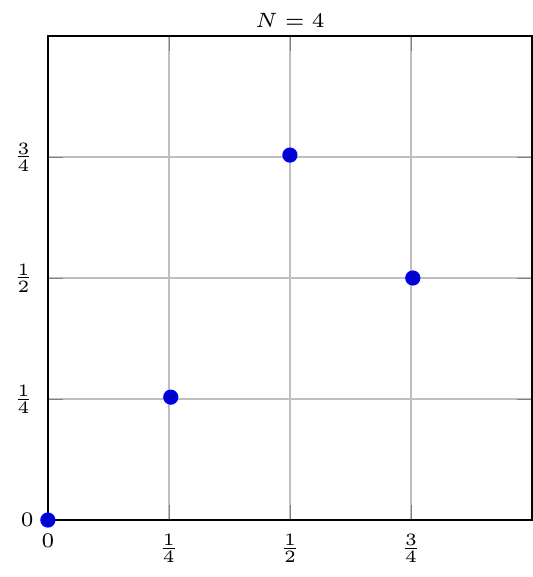} \\
		\includegraphics[width=0.33\linewidth, height=0.32\linewidth]{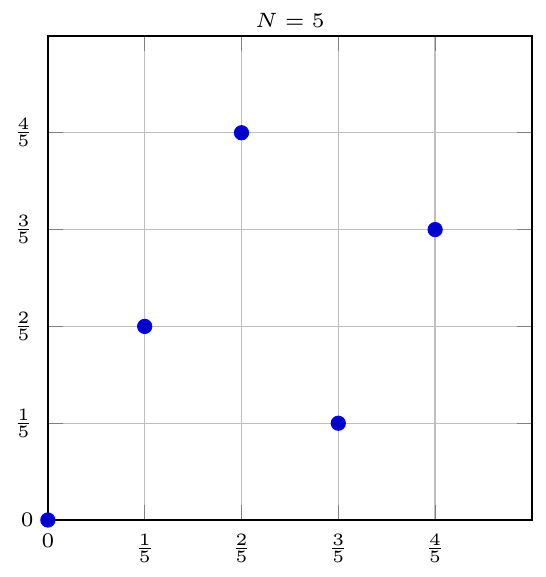} & \includegraphics[width=0.33\linewidth, height=0.32\linewidth]{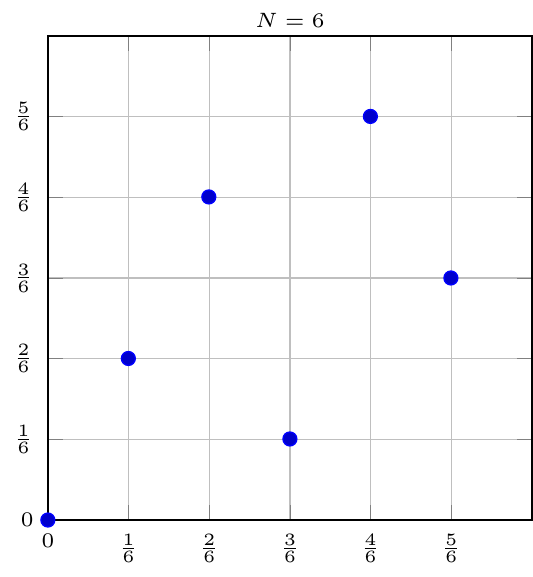} & \includegraphics[width=0.33\linewidth, height=0.32\linewidth]{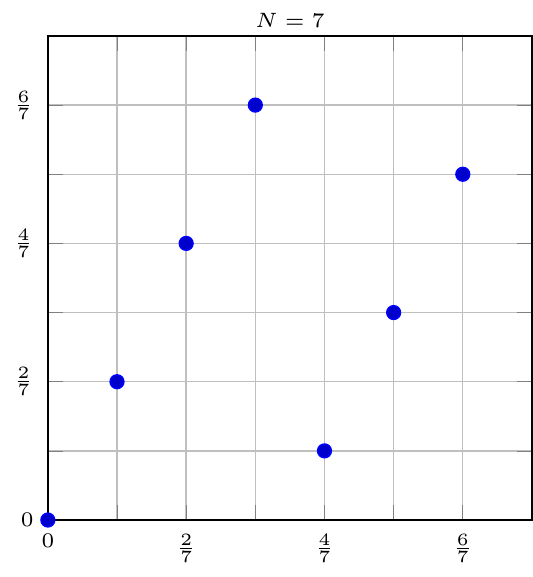} \\
		\includegraphics[width=0.33\linewidth, height=0.32\linewidth]{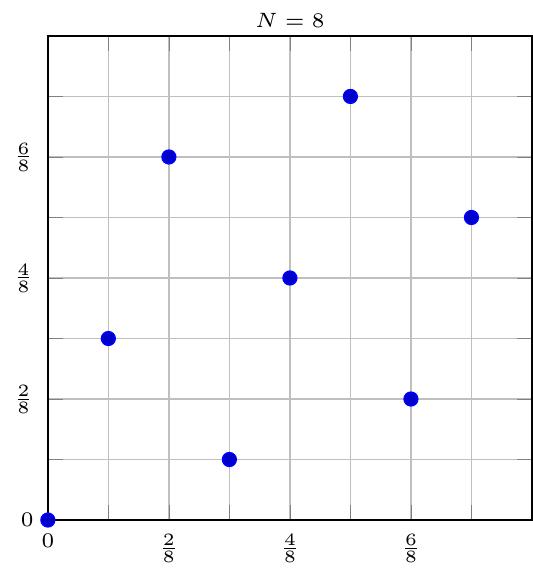} & \includegraphics[width=0.33\linewidth, height=0.32\linewidth]{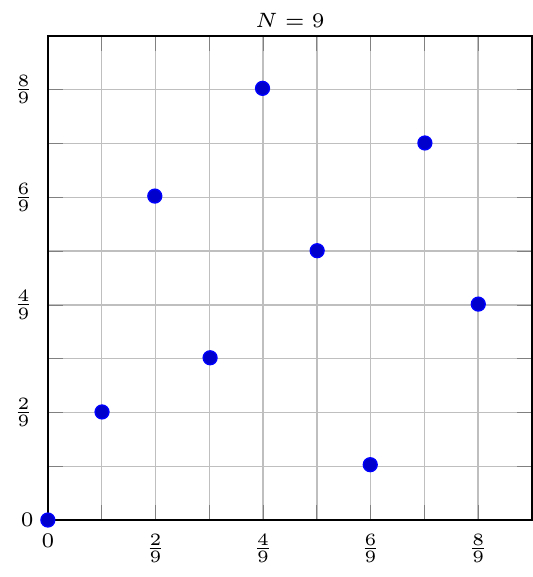} & \includegraphics[width=0.33\linewidth, height=0.32\linewidth]{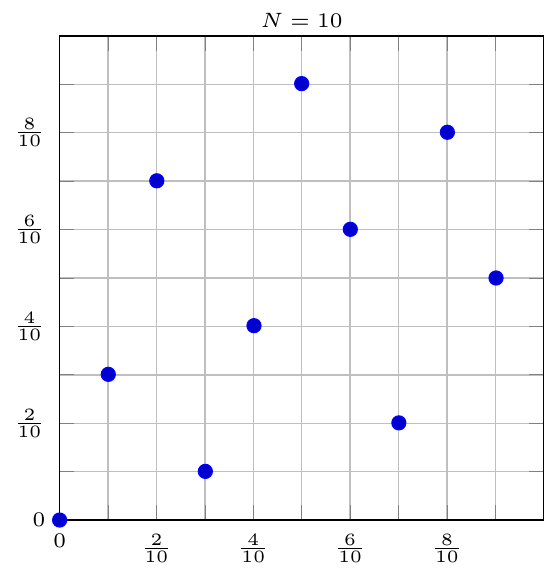} \\
		\includegraphics[width=0.33\linewidth, height=0.32\linewidth]{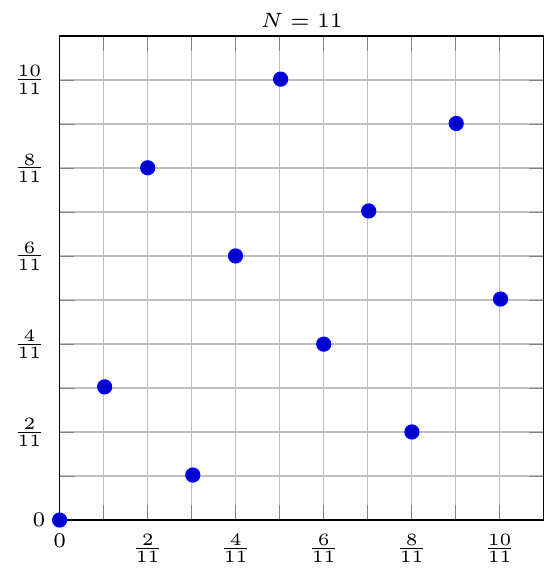} & \includegraphics[width=0.33\linewidth, height=0.32\linewidth]{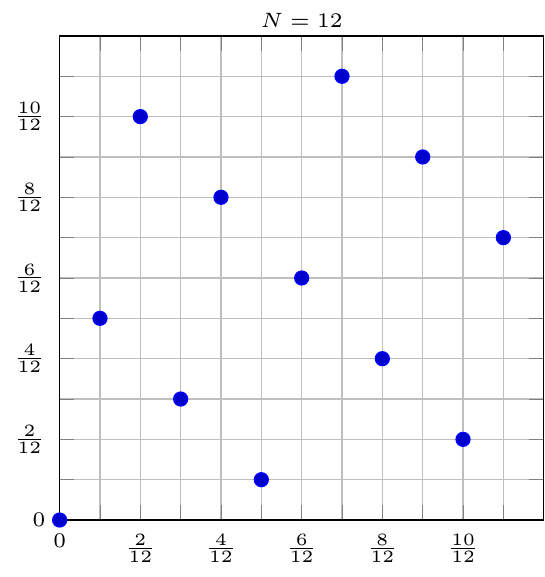} & \includegraphics[width=0.33\linewidth, height=0.32\linewidth]{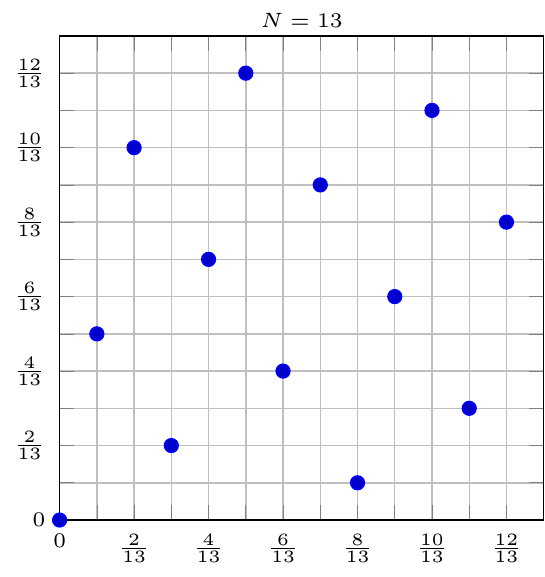} \\
		\includegraphics[width=0.33\linewidth, height=0.32\linewidth]{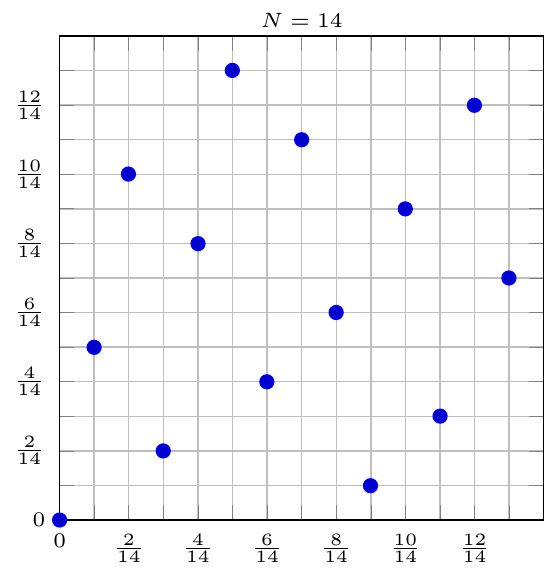} & \includegraphics[width=0.33\linewidth, height=0.32\linewidth]{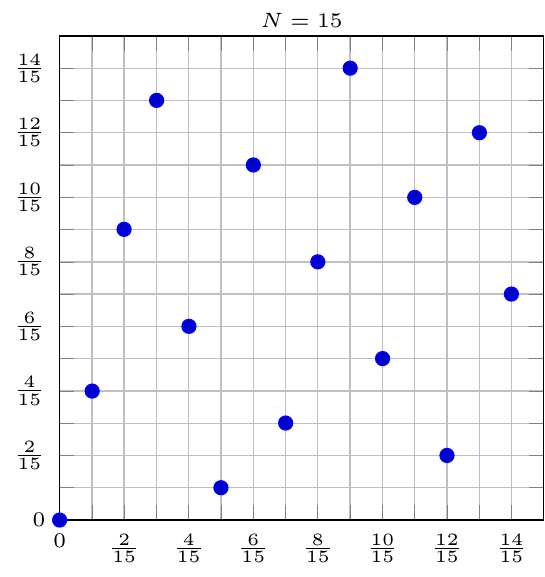} & \includegraphics[width=0.33\linewidth, height=0.32\linewidth]{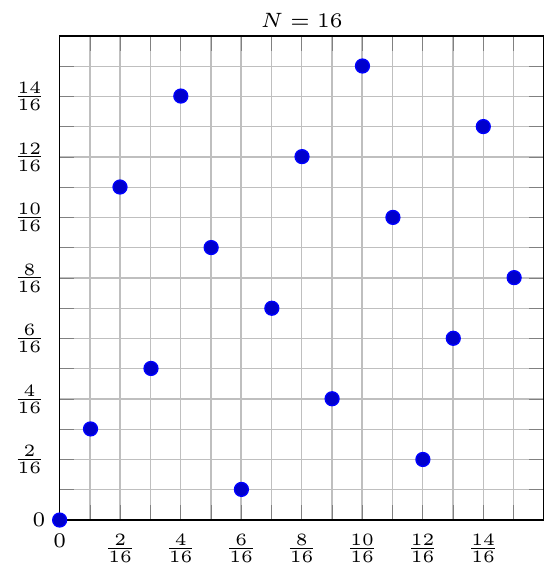} \\
	\end{tabular}
    \caption{Optimal point sets for $N=2, \ldots, 16$ and $\gamma = 1$.} \label{fig_plots1}
\end{figure}

\begin{figure}[t]
	\begin{tabular}{ccc}
		\includegraphics[width=0.33\linewidth, height=0.32\linewidth]{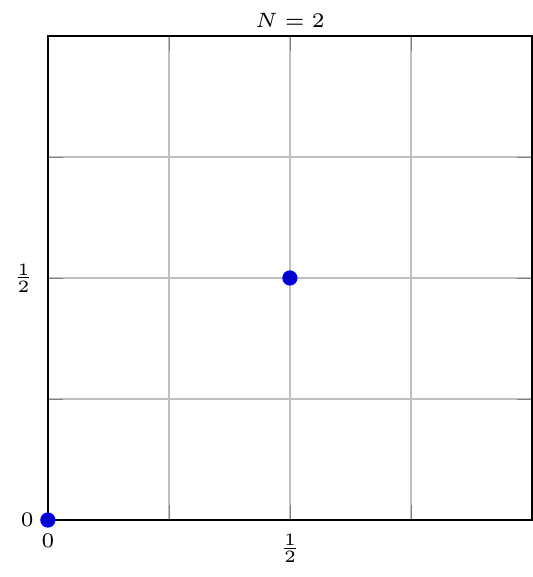} & \includegraphics[width=0.33\linewidth, height=0.32\linewidth]{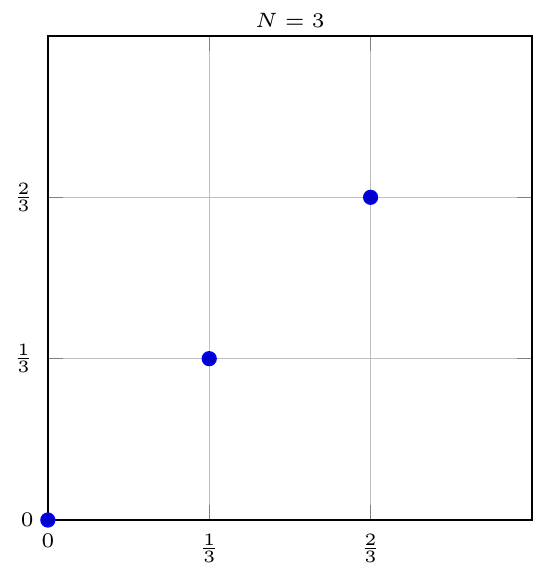} & \includegraphics[width=0.33\linewidth, height=0.32\linewidth]{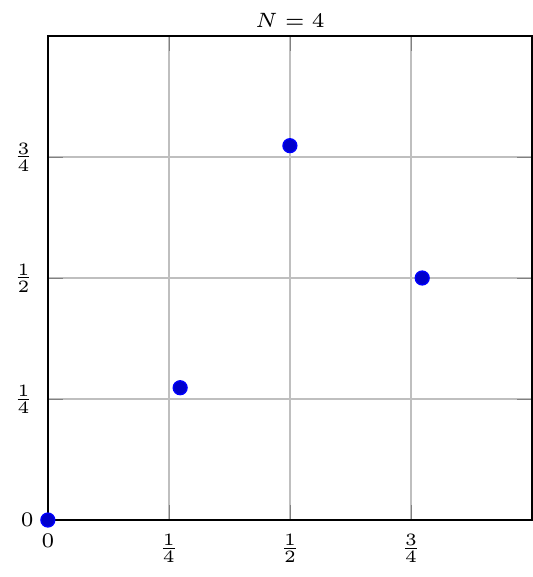} \\
		\includegraphics[width=0.33\linewidth, height=0.32\linewidth]{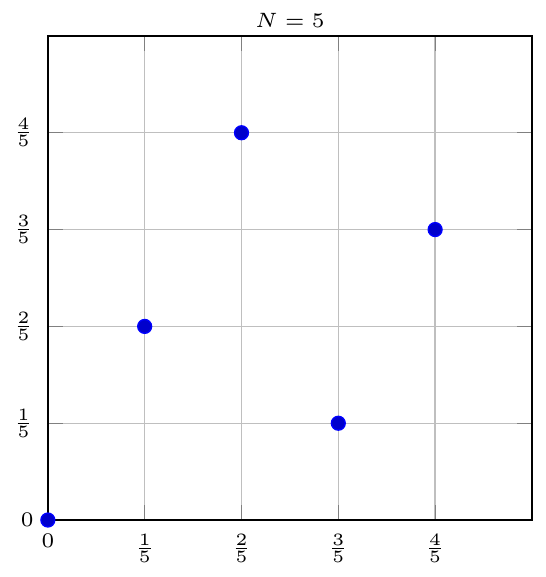} & \includegraphics[width=0.33\linewidth, height=0.32\linewidth]{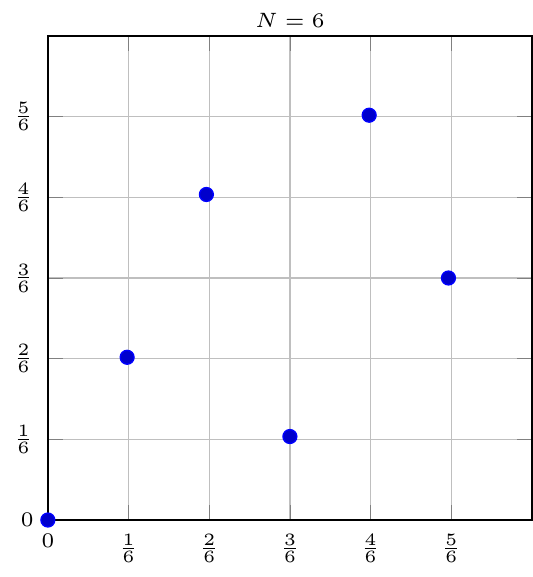} & \includegraphics[width=0.33\linewidth, height=0.32\linewidth]{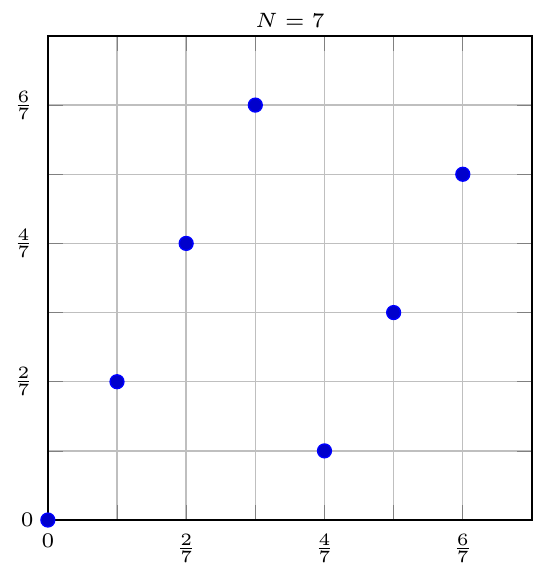} \\
		\includegraphics[width=0.33\linewidth, height=0.32\linewidth]{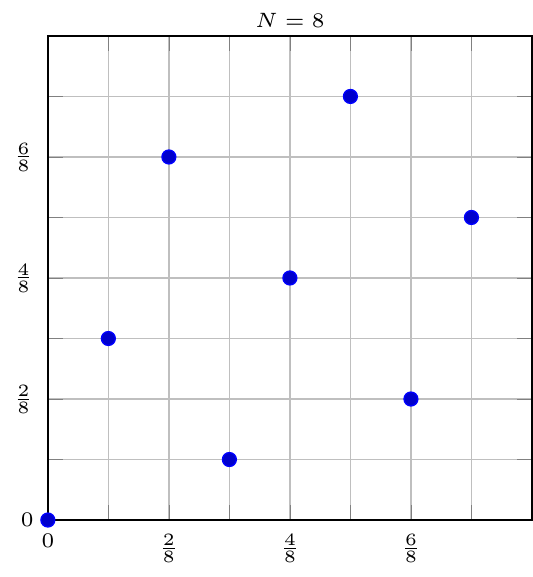} & \includegraphics[width=0.33\linewidth, height=0.32\linewidth]{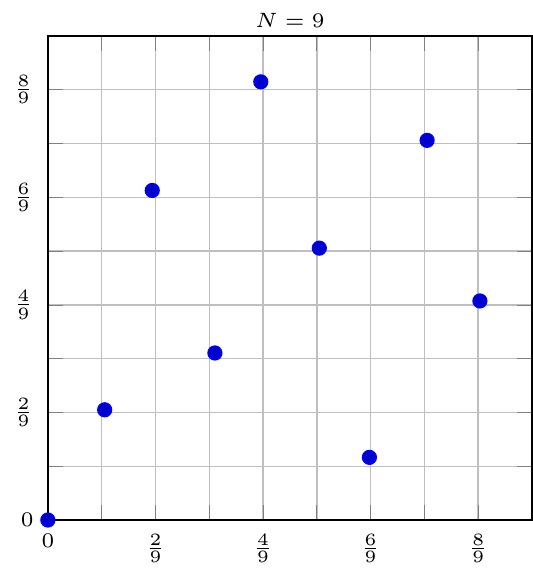} & \includegraphics[width=0.33\linewidth, height=0.32\linewidth]{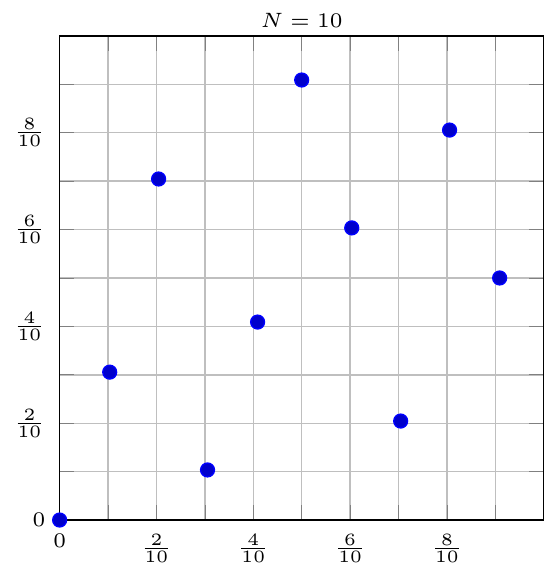} \\
		\includegraphics[width=0.33\linewidth, height=0.32\linewidth]{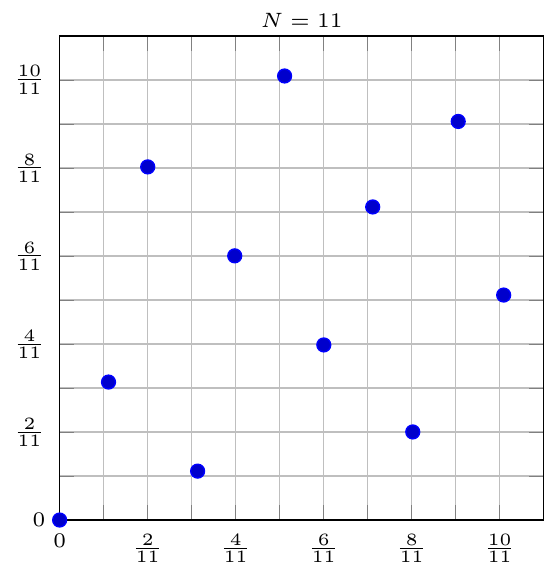} & \includegraphics[width=0.33\linewidth, height=0.32\linewidth]{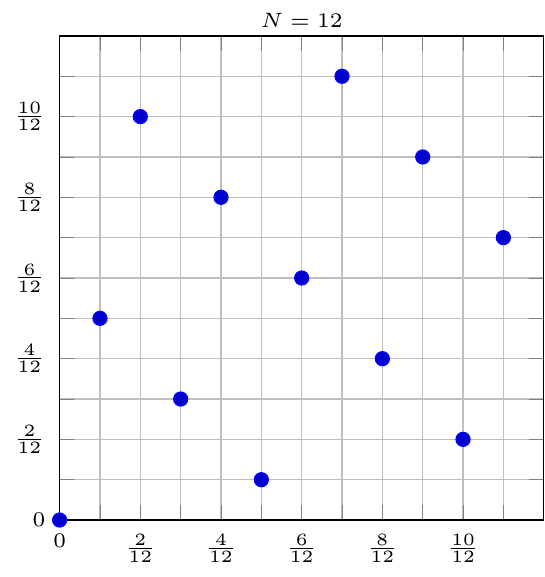} & \includegraphics[width=0.33\linewidth, height=0.32\linewidth]{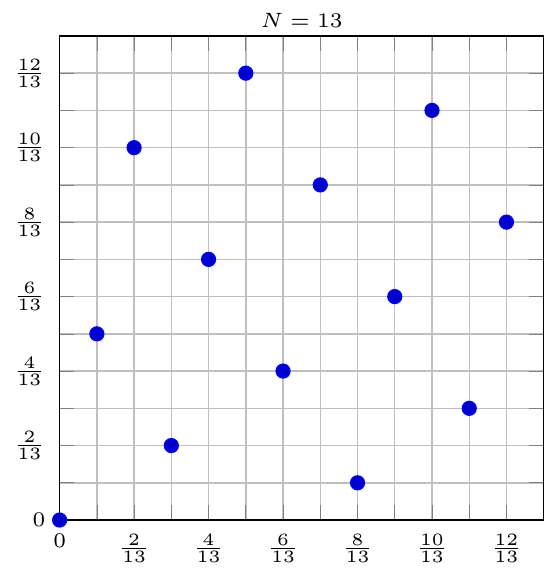} \\
		\includegraphics[width=0.33\linewidth, height=0.32\linewidth]{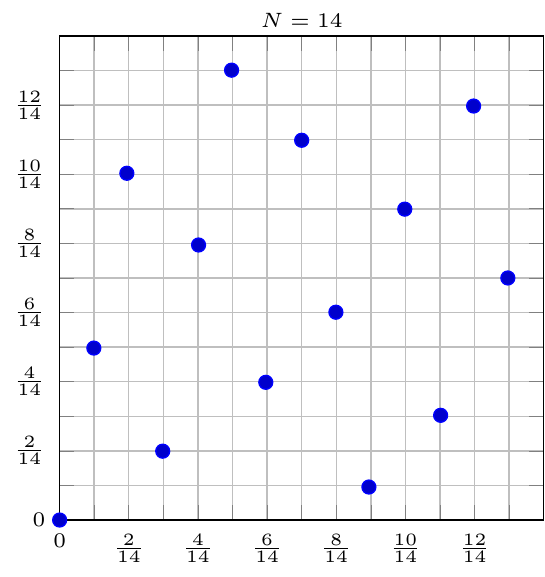} & \includegraphics[width=0.33\linewidth, height=0.32\linewidth]{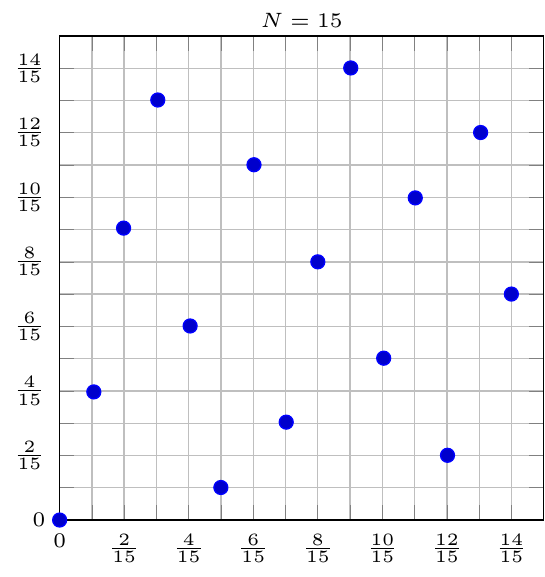} & \includegraphics[width=0.33\linewidth, height=0.32\linewidth]{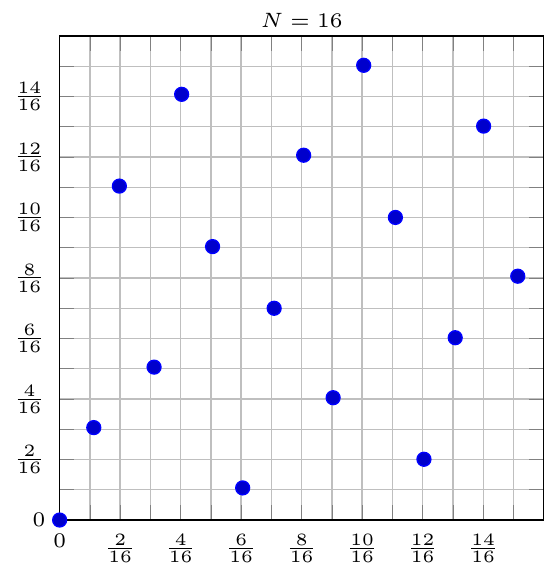} \\
	\end{tabular}
    \caption{Optimal point sets for $N=2, \ldots, 16$ and $\gamma = 6$.} \label{fig_plots2}
\end{figure}

\section{Conclusion}
In the present paper we computed optimal point sets for quasi--Monte Carlo cubature of bivariate periodic functions with mixed smoothness of order one by decomposing the required global optimization problem into approximately $(N-2)! / 2$ local ones. Moreover, we computed lower bounds for each local problem using arbitrary precision rational number arithmetic. Thereby we obtained that our approximation of the global minimum is in fact close to the real solution.

In the special case of $N$ being a Fibonacci number our approach showed that for $N \in \{1,2,3,5,8,13\}$ the Fibonacci lattice is the unique global minimizer of the worst case integration error in $H^1_\text{mix}$. We strongly conjecture that this is true for all Fibonacci numbers. Also in the cases $N=7,12$, the global minimizer is the obtained integration lattice.

In the future we are planning to prove that optimal points are close to lattice points. 
Moreover, we will investigate $H^r_\text{mix}$, i.e. Sobolev spaces with dominating mixed smoothness of order $r\geq 2$ and other suitable kernels and discrepancies.

\begin{acknowledgement}
The authors thank Christian Kuske and Andr\'e Uschmajew for valuable hints and discussions. Jens Oettershagen was supported by the Sonderforschungsbereich 1060 \emph{The Mathematics of Emergent Effects} of the DFG.

\end{acknowledgement}

%

\begin{thebibliography}{99}

\bibitem{Aronszajn_1950} N. Aronszajn: Theory of Reproducing Kernels. Transactions of the American Mathematical Society 68:1950, 337--404.

\bibitem{bezdek87}  J. C. Bezdek, R. J. Hathaway, R. E. Howard, C. A. Wilson, M. P. Windham: Local convergence analysis of a grouped variable version of coordinate descent. J. of Optimization Theory and Applications 54(3):1987, 471--477.

\bibitem{BTY12} D. Bilyk, V. N. Temlyakov, R. Yu: Fibonacci sets and symmetrization in discrepancy theory. J. of Complexity 28:2012, 18--36. 

\bibitem{DickPillich} J. Dick, F. Pillichshammer: {\it Digital Nets and Sequences: Discrepancy Theory and Quasi--Monte Carlo Integration}, Cambridge University Press, 2010. 

\bibitem{Grippo2000} L. Grippo, M. Sciandrone: On the convergence of the block nonlinear Gau\ss–Seidel method under convex constraints. Operations Research Letters 26(3):2000, 127--136.

\bibitem{LP07} G. Larcher, F. Pillichshammer: A note on optimal point distributions in $[0,1)^s$. J. of Computational and Applied Mathematics 206:2007, 977--985.

\bibitem{LuoTseng}  Z. Q. Luo, P. Tseng: On the convergence of the coordinate descent method for convex differentiable minimization. J. of Optimization Theory and Applications  72(1):1992, 7--35

\bibitem{NS84} H. Niederreiter, I. H. Sloan: Integration of nonperiodic functions of two variables by Fibonacci lattice rules. J. of Computational and Applied Mathematics 51:1994, 57--70.

\bibitem{McLachlan} G.J. McLachlan and T. Krishnan. {The EM Algorithm and Extensions}. Wiley series in probability and statistics. John Wiley \& Sons, 1997.

\bibitem{Niederreiter} H. Niederreiter: {\it Quasi-Monte Carlo Methods and Pseudo-Random Numbers}, 	 Society for Industrial and Applied Mathematics. 1987.

\bibitem{NW10} E. Novak, H. Wo\'zniakowski: {\it Tractability of Multivariate Problems. Volume II: Standard Information for Functionals.} European Mathematical Society Publishing House, Z\"urich, 2010.

\bibitem{NW06} J. Nocedal, S.J. Wright: {\it Numerical Optimization}, 2nd edition. Springer, 2006.

\bibitem{Ortega} J. M. Ortega, W. C. Rheinboldt: {\it Iterative Solution of Nonlinear Equations in Several Variables}, Society for Industrial and Applied Mathematics, 1987.

\bibitem{PVC06} T. Pillards, B. Vandewoestyne, R. Cools: Minimizing the $L_2$ and $L_∞$ star discrepancies of a single point in the unit hypercube. J. of Computational and Applied Mathematics 197:2006, 282--285.

\bibitem{SJ94} I. H. Sloan, S. Joe: {\it Lattice Methods for Multiple Integration.} Oxford University Press, New York and Oxford, 1994.

\bibitem{SZ82} V. T. S{\'o}s, S. K. Zaremba: The mean-square discrepancies of some two-dimensional lattices. Studia Scientiarum Mathematicarum Hungarica 14:1982,  255--271.

\bibitem{Temlyakov_1991} V. N. Temlyakov: Error estimates for Fibonacci quadrature formulae for classes of functions. Trudy Mat. Inst. Steklov 200:1991, 327--335.

\bibitem{Tino_Zung_2014} T. Ullrich, D. Zung: {{L}ower bounds for the integration error for multivariate functions with mixed smoothness and optimal {F}ibonacci cubature for functions on the square}. Math. Nachr. 288(7):2015, 743--762.

\bibitem{Uschmajew} A. Uschmajew: Local convergence of the alternating least squares algorithm for canonical tensor approximation. SIAM Journal on Matrix Analysis and Applications  33(2):2012, 639--652.

\bibitem{Wahba75} G. Wahba: Smoothing noisy data with spline functions. Numerische Mathematik 24(5):1975, 383--393.

\bibitem{W77} B.E. White: On optimal extreme-discrepancy point sets in the square. Numerische Mathematik 27: 1977, 157--164.

\bibitem{Zin1997} P. Zinterhof: \"{U}ber einige {A}bsch\"atzungen bei der {A}pproximation von {F}unktionen mit {G}leichverteilungsmethoden. {\"Osterreich. Akad. Wiss. Math.-Naturwiss. Kl. S.-B. II} 185:1976, 121--132.

\end{thebibliography}
%

\end{document}